\documentclass[12pt,reqno]{amsart}

\usepackage[backend=bibtex, style=alphabetic, sorting=nty, maxnames=10]{biblatex}
\usepackage{amssymb,latexsym,amsmath,amsthm,amsfonts,enumerate}
\usepackage{xcolor}
\usepackage[all]{xy}
\usepackage{tikz}

\usepackage[linktocpage,allcolors=purple,colorlinks]{hyperref}

\usepackage{amsfonts}
\usepackage{graphicx}
\usepackage{amsthm}

\usepackage{bm}
\usepackage{setspace}
\usepackage{enumerate}
\usepackage{mathtools}
\usepackage[all]{xy}
\usepackage{latexsym}

\usepackage{tikz}
\usepackage{transparent}
\usepackage{graphics}
\usepackage{float}
\usepackage{mathdots}
\usepackage{xcolor}

\def\Z{{\mathbb Z}}

\definecolor{ForestGreen}{RGB}{34,139,34}

\newtheorem{theorem}{Theorem}[section]
\newtheorem{lemma}[theorem]{Lemma}
\newtheorem{proposition}[theorem]{Proposition}
\newtheorem{cor}[theorem]{Corollary}
\newtheorem*{slemma}{Lemma}

\theoremstyle{definition}
\newtheorem{defi}[theorem]{Definition}
\newtheorem{ex}[theorem]{Example}

\newtheorem{rem}[theorem]{Remark}

\def\F{\mathcal{F}}

\author{Juan Pablo Maldonado}

\address{Departmento de Matematicas - CEMIM - Universidad Nacional de Mar del Plata
\\
Deán Funes 3350, B7602AYL Mar del Plata, Provincia de Buenos Aires, Argentina}

\thanks{The author was supported by CONICET PhD. grant at the Universidad Nacional de Mar del Plata and by the School of Mathematics at the University of Leeds.}

\email{jpemaldonado@mdp.edu.ar}

\title{Frieze matrices and  friezes with coefficients}

\begin{document}

\maketitle

\begin{abstract}
Frieze patterns are combinatorial objects that are deeply related to cluster theory. Determinants of frieze patterns arise from triangular regions of the frieze, and they have been considered in \cite{BroCI, BM}. In this article, we introduce a new type of matrix for any infinite frieze pattern. This approach allows us to give a new proof of the frieze determinant result given by Baur-Marsh.

\end{abstract} \vspace{3mm}

\section{Introduction}

A frieze pattern is an arrangement of numbers that classically starts with a row of zeros followed by a row of ones and ends with a row of ones followed by a row of zeros, and such that every diamond formed by neighbouring entries satisfies the so-called ``diamond rule''.  These arrangements were introduced by Coxeter in \cite{C} and studied by Conway and Coxeter in \cite{CC1,CC2}. Lately, friezes have been actively studied in connection to cluster theory, in such a way that the entries of the frieze are interpreted as the cluster variable of a cluster algebra of type $A$. In this setting the notion of a frieze pattern can be generalized, in particular to infinite friezes (as in \cite{BPM}) or friezes with coefficients (as in \cite{CHJ}). 

The study of symmetric matrices arising from finite frieze patterns was firstly developed in \cite{BroCI}. The main result is a formula for the determinant of a symmetric matrix whose entries form a fundamental region of a finite frieze pattern of positive integers. See Corollary \ref{Cor:Theorem Conway-Coxeter} for details. Afterwards, Baur and Marsh proposed in \cite{BM} a new interpretation drawing upon the cluster algebra setting, and considering a symmetric matrix whose lower part is a fundamental region of a finite frieze pattern with coefficients.  In \cite{B} the author asks for an analogous formula for the determinant of a matrix whose entries are cluster variables of a cluster algebra of type $D$, which was provided by Lampe in {\cite[Theorem 3.6]{Lam}}.

 In this work we provide a new proof of {\cite[Theorem 2.1]{BM}} using a different approach, dropping the use of triangulations. 
 For a symmetric matrix $M$, our main strategy to prove Theorem \ref{Theo: Determinant} is to study an upper triangular matrix $T_M$ which is equivalent to $M$, and reduce the computation of $det(M)$ to that of the determinant of $T_M$.

After our preprint has been posted, work of Holm and J\o rgensen has appeared which includes a more general result, implying Baur-Marsh's frieze determinant, 
see~\cite[Section 4.3]{HJ}.

The structure of this paper is as follows. In Section \ref{SecTh} we  define frieze matrices and we enunciate the main results, giving the proof of our main result,  Theorem \ref{Theo: Determinant}. We finish this section showing two identities fulfilled by the entries of the matrices that we study. Some results of Section  \ref{SecTh} are left to be proved in Appendix \ref{SecProof} in order to ease the reading; thence, Appendix \ref{SecProof} is a section primarily intended to contain demonstrations left in Section \ref{SecTh}, together with some lemmas needed for this purpose. The reader is warned that in some proofs of Section \ref{SecTh} the author may use results from Appendix \ref{SecProof}. \vspace{3mm}

\section{Frieze matrices}\label{SecTh} \vspace{3mm}

For the rest of this article $R$ will denote an integral domain of characteristic zero and $n$ will be a positive integer. \vspace{3mm}

\begin{defi}\label{defi:frieze matrix}  
A symmetric matrix $M = (m_{i,j}) \in frac(R)^{n \times n}$  will be called a \emph{frieze matrix} if  $m_{i,j} = 0$ if and only if $i=j$ and the entries satisfy the \emph{generalized diamond rule} 
\begin{equation}\label{diam rule}
m_{i,j}m_{i+1,j+1} - m_{i+1,j}m_{i,j+1} = m_{i,i+1}m_{j,j+1} 
\end{equation}

 for all $1 \leq i \leq n-1$ and $2 \leq i+1 \leq j \leq n-1$.
\end{defi}Note that Equation \ref{diam rule} says nothing about entries $m_{i,i+1}$ and $m_{i,i+2}$. It will be useful to denote them as $x_{i}$  and $y_{i}$ respectively. $M$ is fully determined by these entries and the repeated application of the generalized diamond rule. \vspace{3mm}

In the literature the relation 
$m_{i,k}m_{j,l} = m_{i,j}m_{k,l} + m_{i,l}m_{j,k}$ for $i \leq j \leq k \leq l$ is called \emph{Ptolemy relation} \cite{P, Pen} or \emph{Plücker relation} \cite{BM}. We will denote this equation as $E_{i,j,k,l}$.

With this notation, in Definition \ref{defi:frieze matrix} we ask the entries of $M$ to fulfill the equation $E_{i,i+1,j,j+1}$ for every pair of indices $(i,j)$ such that $2 \leq i+1 \leq j \leq n-1$. In the next lemma  we see that this is enough to ensure that the entries of $M$ indeed fulfill the Ptolemy relation $E_{i,j,k,l}$ for all quadruples of indices $1 \leq i \leq j \leq k \leq l \leq n$. \vspace{3mm}
 
For completeness we include a proof of the following lemma in Section \ref{SecProof}. Note that this property also appears in the context of finite friezes in \cite[Theorem 3.3]{CHJ}. 
 
\vskip.5cm

\begin{lemma}\label{lm:ptolemy} 
Let $M= (m_{i,j})$ be a frieze matrix in $frac(R)^{n \times n}$. Then \begin{equation*}
m_{i,k}m_{j,l} = m_{i,j}m_{k,l} + m_{i,l}m_{j,k} 
\end{equation*}
for every $1 \leq i \leq j \leq k \leq l \leq n$. 
\end{lemma}

The main feature that we will use to compute the determinant of $M$ is a triangulated form. For this, denote by $T_M  \in frac(R)^{n \times n}$ the upper triangular matrix whose entries are given by:

\[
t_{i,j} = \left\{ \begin{array}{l c l}
	m_{2,j} & & \text{if} \ i=1 \\[0.3em]
	m_{1,j} & & \text{if} \ i=2 \\[0.4em]
	0 & & \text{if} \ i\geq3 \ \land \ j < i \\[0.3em]
	\dfrac{-2m_{1,j}}{m_{1,i-1}}m_{i-1,i} & & \text{if} \ i\geq3 \ \land \ j \geq i
\end{array}\right.
\] \vspace{2mm}

\begin{proposition}\label{Prop:Triangulated form} 
If $M = (m_{i,j}) \in frac(R)^{n \times n}$ is a frieze matrix, then it is row equivalent to the upper triangular matrix T$_{M}$ defined above, and $det(T_M) = -det(M)$.
\end{proposition} \vspace{3mm}

We will prove Proposition \ref{Prop:Triangulated form} in Section \ref{SecProof}.

\begin{ex}\label{exM}
For the frieze matrix $$
M= \left( \begin{array}{c c c c c c}
0 & 1 & 2 & 2 & -1 & 5 -\frac{\sqrt{5}}{2}\\[0.5ex]
1 & 0 & -2 & 1 & \frac{1}{2} & \frac{-7}{2}+\frac{\sqrt{5}}{4}\\[0.5ex]
2 & -2 & 0 & 6 & -1 & 3 - \frac{\sqrt{5}}{2}\\[0.5ex]
2 & 1 & 6 & 0 & 2 & \sqrt{5} \\[0.5ex]
-1 & \frac{1}{2}  & -1 & 2 & 0 & 1 \\[0.5ex]
5 -\frac{\sqrt{5}}{2} & \frac{-7}{2}+\frac{\sqrt{5}}{4} & 3 - \frac{\sqrt{5}}{2} & \sqrt{5} & 1  & 0 \\[0.5ex]
\end{array}\right)$$ we have that

\[T_M = \left( \begin{array}{c c c c c c}
 1 & 0 & -2 & 1 & \frac{1}{2} & \frac{-7}{2}+\frac{\sqrt{5}}{4} \\[0.5em]
 0 & 1 & 2 & 2 & -1 & 5 -\frac{\sqrt{5}}{2}\\[0.5em]
 0 & 0 & 8 & 8 & -2 & 20-2\sqrt{5} \\[0.5em]
 0& 0 & 0 & -12 & 4 & 3\sqrt{5} -18 \\[0.5em]
 0& 0 & 0 & 0 & \frac{11}{6} & \sqrt{5} -6 \\[0.5em]
 0& 0 & 0 & 0 & 0 & \frac{17}{11}\sqrt{5}-\frac{173}{22} \\
 
\end{array}\right)\]
\end{ex} \vspace{6mm}

The main result of this section follows directly from Proposition \ref{Prop:Triangulated form} as the determinant of $M$ can computed  using the upper triangular matrix $T_M$. 
\vspace{-1mm}
	
\begin{theorem}\label{Theo: Determinant} If $M$ is a frieze matrix then 
$Det(M) = -(-2)^{n-2}m_{1,n}\prod\limits_{i=1}^{n-1}x_{i}$ 
\end{theorem} \vspace{0mm}
	
\begin{proof}
As $Det(M) = - Det(T_M)$, and this last determinant can be computed as the product of the entries in the diagonal of $T_M$ we have that \vspace{-3mm} 

\begin{eqnarray*}
&& \hspace{-7mm} Det(M) = - Det(T_{M}) = - \prod\limits_{i=1}^{n}t_{i,i} = - t_{1,1}t_{2,2}\prod\limits_{i=3}^{n}t_{i,i} =-m_{1,2}m_{1,2}\prod\limits_{i=3}^{n}\dfrac{-2m_{1,i}}{m_{1,i-1}}m_{i-1,i}  \\
&=& -(-2)^{n-2}m_{1,2}m_{1,2}\dfrac{m_{1,n}}{m_{1,2}} \prod\limits_{i=3}^{n}m_{i-1,i} = -(-2)^{n-2}m_{1,n} \prod\limits_{i=2}^{n}m_{i-1,i} = -(-2)^{n-2}m_{1,n} \prod\limits_{i=1}^{n-1}x_{i}
\end{eqnarray*}
\end{proof}

Two results that we recover from Theorem~\ref{Theo: Determinant} are stated in Corollary~\ref{Cor:Theorem Conway-Coxeter} and \ref{Cor: Theorem Baur-Marsh}, so we recover \cite[Theorem 4]{BroCI} and  \cite[Theorem 1.1]{BM} respectively. \vspace{3mm}

To give the context of Theorem $4$ in \cite{BroCI} we recall the notion of frieze patterns as first introduced by Conway and Coxeter in \cite{CC1, CC2}. For further details we refer to \cite{B, M-G}. An array of numbers $\F = (f_{i,j})_{i,j \in \Z}$, with $j \geq i$,  is a \emph{frieze pattern} if the following  holds:
\begin{enumerate}[i)]
\itemsep 0cm
\item $f_{i,i}= 0$ \hspace{3mm} for all $i \in \Z$
\item $f_{i,i+1} = 1$ \hspace{3mm} for all $i \in \Z$
\item $f_{i,j}f_{i+1,j+1} - f_{i+1,j}f_{i,j+1} = 1$ \hspace{3mm} for all $i \leq j \in \Z$ 
\end{enumerate}

Usually the entries of $\F$ are displayed in rows, shifted with respect to each other.The frieze $\F$ is  \emph{finite}  if $f_{i, i+k-1} = 1$ for some fixed $k$ and for all $i \in \Z$. The positive integer $k$ is called the \emph{order} $\F$. A frieze $\F$ is a \emph{frieze pattern of positive integers} if all the $f_{i,j}$ out of the rows of zeros are positive integers. The third row of $\F$, whose elements are of the form $f_{i,i+2}$, is called the \emph{quiddity row} of $\F$ and its entries are noted $a_{i} = f_{i,i+2}$. If $\F$ is finite of order $k$ then it is $k-$periodic ($f_{i,j} = f_{i+k,j+k} \hspace{1mm} \forall i,j$), see \cite{CC1, CC2} problem $(21)$. In this case we call a \emph{quiddity sequence} of $\F$ the sequence of numbers $(a_{1}, \ldots, a_{k})$. Finally, a \emph{fundamental region} for a finite integral frieze pattern $\F$ is given by the elements of the form $f_{i,j}$, with $1 \leq i \leq k$ and $i \leq j \leq k$. The main theorem stated in \cite{BroCI} is the following corollary of Theorem \ref{Theo: Determinant}. \vspace{2mm}

\begin{spacing}{1.1}
\begin{cor}{\cite[Theorem 4]{BroCI}}\label{Cor:Theorem Conway-Coxeter} Let $\mathcal{F}$ be a finite integer frieze pattern of order $k$, with quiddity sequence $(a_1, \ldots, a_k)$. Let us define $M_{\F} = (m_{ij}) \in \Z^{k \times k}$ as the symmetric matrix whose lower part is given by the fundamental region of $\F$ (i.e. $m_{i,j} = f_{i,j}$ if $1 \leq i \leq j \leq k$ and $m_{i,j} = m_{j,i}$ if $1 \leq j < i \leq n $) Then Det($M_{\F})=-(-2)^{k-2}$.
\end{cor}
\end{spacing} \vspace{3mm}

\begin{proof}
If in Definition \ref{defi:frieze matrix} we set $n=k$, $x_{i} = 1$ for all $i \in [1, \ldots, k-1]$ and $y_{i} = a_{i}$ for all $i \in [1, \ldots, k-2]$ we recover the matrix $M_{\F}$, so we see that $M_{\F}$ is a frieze matrix. Then, by Theorem~\ref{Theo: Determinant}, $Det(M_{\F}) = -(-2)^{k-2}m_{1,k}$, since all the $x_{i}$ are equal to one. Besides, as $\F$ is of order $k$, $m_{1,k} = f_{1,k} = 1$. Therefore, $Det(M_{\F}) = -(-2)^{k-2}$.
\end{proof}

\begin{spacing}{1.3}
Now we proceed to give the proof of \cite[Theorem 1.1]{BM} in terms of our Theorem~\ref{Theo: Determinant}. For this, consider a $2 \times n$ matrix  $X= \left(\begin{array}{c c c c}
a_{1} & a_{2} & \ldots & a_{n} \\
b_{1} & b_{2} & \ldots & b_{n}
\end{array}\right)$ whose entries are indeterminate. Denote by $\Delta_{ij} = \left|\begin{array}{c c}
a_{i} & a_{j} \\
b_{i} & b_{j}
\end{array}\right|$ the minor of $X$ given by the columns $i,j$ and let $A=(A_{ij})$ be the matrix such that $A_{ij} = \left\{\begin{array}{l c l}
\Delta_{ij} & & i \geq j \\
\Delta_{ji} & & i < j
\end{array}
\right.$
\end{spacing}

The authors showed in \cite[Theorem 2.1]{BM} that the entries of $A$, fulfill the Ptolemy relation; in particular this holds for $i < i+1 < i+k-1 < i+k $ (with $k \geq 3$). So we have that $A$ is a frieze matrix, and we can apply Theorem~\ref{Theo: Determinant} to obtain the following immediate corollary of Theorem \ref{Theo: Determinant}.

\begin{cor}{\cite[Theorem 1.1]{BM}}\label{Cor: Theorem Baur-Marsh} $Det(A) = -(-2)^{n-2}\Delta_{1n}\prod\limits_{i=1}^{n-1}\Delta_{i(i+1)}$
\end{cor} \vspace{5mm}

We finish this section with two results giving identities in the entries of the matrices we have studied. The first one provides a formula to compute the entries of a frieze matrix $M$ only knowing the entries in its first two rows and the elements $m_{i,i+1} \in frac(R)$. The second one proves that the entries of the triangular form $T_M$ of a frieze matrix $M$ fulfill an analogous formula of the generalized diamond rule in Equation \ref{diam rule}. \vspace{3mm}

\begin{proposition}\label{Cor:Diagonals}
For $3 \leq i \leq n-1$ and $j \geq i$ it holds that 
\[
m_{i,j} = \frac{m_{1,i}m_{2,j}}{m_{1,2}} + \frac{m_{2,i}m_{1,j}}{m_{1,2}} - 2 \sum\limits_{t=3}^{i}\dfrac{m_{1,i}m_{1,j}}{m_{1,t}m_{1,t-1}}m_{t-1,t}
\]
\end{proposition} \vspace{3mm}

\begin{proof}
We will treat the cases $i=3$ and $4 \leq i \leq n+1$ separately.\vspace{3mm}
	
If $i=3$ and $j \geq 3$ we have by Lemma~\ref{lm:ptolemy} that 
\[
m_{1,2}m_{3,j} = m_{1,3}m_{2,j} - m_{2,3}m_{1,j} = m_{1,3}m_{2,j} + m_{2,3}m_{1,j} - 2m_{2,3}m_{1,j} 
\]	 \vspace{-5mm}
	
Consider now $4 \leq i \leq n+1$ and fix $k \geq i$. Due to the proof of Lemma~\ref{Lemma:Mk} we know that \vspace{-3mm}
	
\[
m_{i,j}^{k} = -2\dfrac{m_{1,j}}{m_{1,i-1}}m_{i-1,i}
\]	\vspace{-3mm}
	
But by definition of $m_{i,j}^{k}$ this element is equal to $m_{i,j}^{2} - \sum\limits_{t=3}^{\min\{i-1,k\}}\dfrac{m_{1,i}}{m_{1,t}}m_{t,j}^{t-1}$. As \\ $\min\{i-1,k\} = i-1$ it turns out that \vspace{-3mm}
	
\[
m_{i,j} -\frac{m_{1,i}m_{2,j}}{m_{1,2}} -\frac{m_{2,i}m_{1,j}}{m_{1,2}} - \sum\limits_{t=3}^{i-1}\dfrac{m_{1,i}}{m_{1,t}}m_{t,j}^{t-1} = -2\dfrac{m_{1,j}}{m_{1,i-1}}m_{i-1,i}
\] \vspace{-1mm}
	
Writing $m_{t,j}^{t-1} = -2\dfrac{m_{1,j}}{m_{1,t-1}}m_{t-1,t}$ and $-2\dfrac{m_{1,j}}{m_{1,i-1}}m_{i-1,i} = -2\dfrac{m_{1,i}m_{1,j}}{m_{1,i}m_{1,i-1}}m_{i-1,i}$ the proof is completed.
\end{proof} \vspace{4mm}

A natural question that arises while studying the matrices $T_M$ is if they are frieze matrices; i.e. if they fulfill Equation (\ref{diam rule}). The reader may check in Example \ref{exM} that the entries of $T_M$ do not fulfill the generalized diamond rule. Despite of that, one can observe that a different rule holds: the determinant of any $2 \times 2$ matrix formed by neighbouring entries above the diagonal is equal to zero. The following proposition states this for every matrix $T_M$.

\begin{proposition}\label{Prop:Diamond zero}
Let $M$ be a frieze matrix and T$_{M}$ its triangulated form given in Proposition~\ref{Prop:Triangulated form}. Then 
\begin{enumerate}[a)]
\itemsep 0cm	
\item $t_{i,j}t_{i+1,j+1} - t_{i+1,j}t_{i,j+1} = 0 \hspace{3mm} \text{for all} \hspace{3mm} i \geq 2 \hspace{3mm} \text{and} \hspace{3mm} j \geq i+1$.
			
\item $t_{i,i}t_{i+1,i+1} + 2m_{i,i+1}t_{i,i+1} = 0 \hspace{3mm} \text{for all} \hspace{3mm} i \geq 2.$
\end{enumerate}
\end{proposition} \vspace{3mm}

\begin{proof}
\begin{enumerate}[a)]
\item We will treat the cases $i=2$ and $i \geq 3$ separately. \vspace{3mm}
			
First, if $i=2$ and $j \geq 3$ then 
\[
t_{2,j}t_{3,j+1} - t_{3j}t_{2(j+1)} = m_{1,j}\left(-2\frac{m_{1,j+1}}{m_{1,2}}m_{2,3}\right) - \left(-2\frac{m_{1,j}}{m_{1,2}}m_{2,3}\right)m_{1,j+1} = 0
\]

Secondly, if $i \geq 3$ and $j \geq i+1$ we have that \vspace{-5mm}

\begin{align*}
\begin{split}
& t_{i,j}t_{i+1,j+1} -t_{i,j+1}t_{i+1,j} = \frac{-2m_{1,j}m_{i-1,i}}{m_{1,i-1}}\frac{-2m_{1,j+1}m_{i,i+1}}{m_{1,i}}-\frac{-2m_{1,j+1}m_{i-1,i}}{m_{1,i-1}}\frac{-2m_{1,j}m_{i,i+1}}{m_{1,i}}
\\  & =\frac{4m_{1,j}m_{i-1,i}m_{1,j+1}m_{i,i+1}}{m_{1,i-1}m_{1,i}} - \frac{4m_{1,j+1}m_{i-1,i}m_{1,j}m_{i,i+1}}{m_{1,i-1}m_{1,i}} = 0
\end{split}
\end{align*} 
			
\item Again we have to treat the cases $i=2$ and $i \geq 3$ separately.

If $i=2$ \vspace{-.8cm}

\begin{align*}
\begin{split}
t_{2,2}t_{3,3} + 2m_{2,3}t_{2,3} = m_{1,2}\frac{-2m_{1,3}m_{2,3}}{m_{1,2}} + 2m_{2,3}m_{1,3} = 0
\end{split}
\end{align*}

And if $3 \geq i \geq n$ \vspace{-.7cm}

\begin{align*}
\begin{split}
& t_{i,i}t_{i+1,i+1} + 2m_{i,i+1}t_{i,i+1} = \frac{-2m_{1,i}m_{i-1,i}}{m_{1,i-1}}\frac{-2m_{1,i+1}m_{i,i+1}}{m_{1,i}} + 2m_{i,i+1}\frac{-2m_{1,i+1}m_{i-1,i}}{m_{1,i-1}} = \\[0.4em]
& 4\frac{m_{i-1,i}m_{1,i+1}m_{i,i+1}}{m_{1,i-1}} -4\frac{m_{i,i+1}m_{1,i+1}m_{i-1,i}}{m_{1,i-1}} = 0
\end{split}
\end{align*}
\end{enumerate}
\end{proof}

\section{Appendix: Proofs}\label{SecProof}

The goal of this section is to prove Proposition \ref{Prop:Triangulated form}. In preparation, we first prove Lemma~\ref{lm:ptolemy} and then 
show an 
additional Lemma on auxiliary matrices $M_k$ which are needed for 
the proof of the proposition. 

\vspace{3mm}

\begin{slemma} [Lemma~\ref{lm:ptolemy}]
Let $M= (m_{i,j})$ be a frieze matrix in $frac(R)^{n \times n}$. Then \begin{equation}
m_{i,k}m_{j,l} = m_{i,j}m_{k,l} + m_{i,l}m_{j,k}  \tag{$E_{i,j,k,l}$} \label{Ptolemy}
\end{equation}
for every $1 \leq i \leq j \leq k \leq l \leq n$.
\end{slemma}

To clarify the following arguments, whenever an equality holds because of an equation $E_{i,j,k,l}$ we will indicate this by writing the tag with the corresponding indices on the right.

\begin{proof}[Proof of Lemma~\ref{lm:ptolemy}]
Firstly, if one of the
inequalities between the indices $i; j; k; l$ in
Lemma~\ref{lm:ptolemy} 
is an equality, 
then the equation (\ref{Ptolemy}) is trivial. \vspace{1mm}
	
Suppose now that $i<j<k<l$. We will prove the assertion by induction on $d=l-i$, the distance between the first and last subscript. The minimum non-trivial distance for $i$ and $l$ is $l-i = 3$, which implies that $j=i+1$ and $k=i+2$. Therefore the right hand side is \vspace{-5mm}

\begin{equation*}
m_{i,i+1}m_{i+2,i+3} + m_{i,i+3}m_{i+1,i+2}  = x_{i}x_{i+2} + m_{i,i+3}x_{i+1}  
\end{equation*} by Equation \ref{diam rule} this last element is equal to
\[ x_{i}x_{i+2} + \left(\dfrac{y_{i}y_{i+1}-x_{i}x_{i+2}}{x_{i+1}}\right)x_{i+1} = y_{i}y_{i+1} = m_{i,i+2}m_{i+1,i+3} = m_{i,k}m_{j,l} \]

Now, assume that $E_{i^{\prime},j^{\prime},k^{\prime},l^{\prime}}$ holds for all $ i^{\prime} < j^{\prime} < k^{\prime} < l^{\prime}$ with $l^{\prime}-i^{\prime} \leq d$. Consider $i<j<k<l$ with $l-i = d+1$. Then, since $l-i \geq 3$, by the generalized diamond rule we have that  \vspace{-5mm}
	
\begin{align}
& \hspace{-7mm} m_{i,j}m_{k,l} + m_{i,l}m_{j,k} = m_{i,j}m_{k,l} + \left(\dfrac{m_{i,l-1}m_{i+1,l}-m_{i,i+1}m_{l-1,l}}{m_{i+1,l-1}}\right)m_{j,k} \notag \\[0.5em] 
& =m_{i,j}m_{k,l} +  \dfrac{(m_{i,l-1}m_{j,k})m_{i+1,l}-m_{j,k}m_{i,i+1}m_{l-1,l}}{m_{i+1,l-1}} \notag \\[0.5em]
& =\dfrac{m_{i,j}m_{k,l}m_{i+1,l-1}+ (m_{i,l-1}m_{j,k})m_{i+1,l}-m_{j,k}m_{i,i+1}m_{l-1,l}}{m_{i+1,l-1}} \notag
\end{align}
As $E_{i,j,k,l-1}$ holds, this last term is equal to 

\vspace{-5.8mm}\begin{align}
&\dfrac{m_{i,j}m_{k,l}m_{i+1,l-1}+\left(m_{i,k}m_{j,l-1}-m_{i,j}m_{k,l-1}\right)m_{i+1,l}-m_{j,k}m_{i,i+1}m_{l-1,l}}{m_{i+1,l-1}} \hspace{1cm} [E_{i,j,k,l-1}] \notag \\[0.5em] 
& = \dfrac{m_{i,j}(m_{k,l}m_{i+1,l-1}-m_{k,l-1}m_{i+1,l})+m_{i,k}m_{j,l-1}m_{i+1,l}-m_{j,k}m_{i,i+1}m_{l-1,l}}{m_{i+1,l-1}} \notag \\[0.5em]
& = \dfrac{m_{i,j}m_{i+1,k}m_{l-1,l}+m_{i,k}m_{j,l-1}m_{i+1,l} - m_{j,k}m_{i,i+1}m_{l-1,l} }{m_{i+1,l-1}} \hspace{1cm} [E_{i+1,k,l-1,l}] \notag \\[0.5em]
& = \dfrac{(m_{i,j}m_{i+1,k}- m_{j,k}m_{i,i+1})m_{l-1,l}+m_{i,k}m_{j,l-1}m_{i+1,l}}{m_{i+1,l-1}} \notag \\[0.5em]
& = \dfrac{m_{i,k}m_{i+1,j}m_{l-1,l}+m_{i,k}m_{j,l-1}m_{i+1,l}}{m_{i+1,l-1}}  \hspace{1cm} [E_{i,i+1,j,k}] \notag \\[0.5em]
&= \dfrac{m_{i,k}(m_{i+1,j}m_{l-1,l}+m_{j,l-1}m_{i+1,l})}{m_{i+1,l-1}}   =\dfrac{m_{i,k}m_{j,l}m_{i+1,l-1}}{m_{i+1,l-1}} =
m_{i,k}m_{j,l} \hspace{1cm} [E_{i+1,j,l-1,l}] \notag
\end{align} 
\vspace{-1cm}

\end{proof}

\vspace{5mm} Before proving Proposition \ref{Prop:Triangulated form} we introduce auxiliary matrices $M_{0}, M_1, \ldots, M_{n-1}$, being $M_{n-1}$ the desired upper triangular matrix $T_{M}$, as follows. The matrix $M_{0}$ is obtained by swapping the first row of $M$ with its second row, $M_{1}$ is the result of applying the sequence of row operations $``R_{i} - \frac{m_{1,i}}{m_{1,2}}R_{1} \rightarrow R_{i}"$ (for $3 \leq i \leq n$) to $M_{0}$, and $M_{2}$ results from applying the sequence of row operations $``R_{i} - \frac{m_{2,i}}{m_{1,2}}R_{2} \rightarrow R_{i}"$ (for $3 \leq i \leq n$) to $M_{1}$. From there on, the matrix $M_{k}$ is obtained by applying the sequence of row operations $``R_{i} - \frac{m_{1,i}}{m_{1,k}}R_{k} \rightarrow R_{i}"$ (for $k+1 \leq i \leq n$) to the matrix $M_{k-1}$.

For an explicit calculation, let us denote by $m_{i,j}^{k}$ the $ij$-entry of $M_{k}$ (observe that the super index is not a power). We define $m_{i,j}^{k}$ as:

\begin{align*}
m_{i,j}^{0} & = \left\{
\begin{array}{l c l}
m_{2,j} & & \text{if} \ i=1 \\[0.2em]
m_{1,j} & & \text{if} \ i=2 \\[0.2em]
m_{i,j} & & \text{if} \ i \geq 3 
\end{array}\right.  \\[2em]
m_{i,j}^{1} & = \left\{
\begin{array}{l c l}
m_{i,j}^{0} & & \text{if} \ i=1,2 \\[0.5em]
m_{i,j} - \frac{m_{1,i}}{m_{1,2}}m_{2,j} & & \text{if} \ i \geq 3 \end{array}\right. \\[2em]
m_{i,j}^{2} & = \left\{
\begin{array}{l c l}
m_{i,j}^{0} & & \text{if} \ i=1,2 \\[0.5em]
m_{i,j} -  \frac{m_{1,i}}{m_{1,2}}m_{2,j} - \frac{m_{2,i}}{m_{12}}m_{1,j}& & \text{if} \ i \geq 3 
\end{array}\right.
\end{align*} 

And inductively for $k\geq 3$ \vspace{4mm}

\begin{equation}
m_{i,j}^{k} = \left \{\begin{array}{l c l}
m_{i,j}^{k-1} & & \text{if} \ 1\leq i \leq k \\[1em]
m_{i,j}^{k-1} - \frac{m_{1,i}}{m_{1,k}}m_{k,j}^{k-1}& & \text{if} \ k+1 \leq i \leq n
\end{array}\right.
\end{equation}
\vspace{4mm}

Before moving forward with the proof of  Proposition \ref{Prop:Triangulated form} we give several useful observations. \vspace{2mm}

\begin{rem} 
Observe that for $j=1,2$ and $i > j$ the entries  $m_{ij}^{2}$ are all zero. Besides, is not hard to prove by induction on $k$ that an equivalent definition for $m_{i,j}^{k}$, with $k \geq 3$ is 
\end{rem} 

\begin{equation}\label{sum}
m_{i,j}^{k} = \left \{\begin{array}{l l}	m_{i,j}^{2} &  \text{if} \  i \in \{1, 2 ,3\} \\[0.3em]
m_{i,j}^{2} - \sum\limits_{t=3}^{\min\{i-1,k\}}\frac{m_{1,i}}{m_{1,t}}m_{t,j}^{t-1}& \text{if} \ 4 \leq i \leq n
\end{array}\right.
\end{equation} \vspace{2mm}

\begin{lemma}\label{Lemma:Mk} 
For all $k \geq 3$ the entries of $M_{k}$ have the following form:
		
\[
m_{i,j}^{k} = \left \{\begin{array}{l c l} (i) \hspace{3mm}	m_{i,j}^{2} & & \text{if} \  i = 1,2 \\[0.6em]
(ii) \hspace{3mm} 0 & & \text{if} \ i \geq 3 \ \land \ j \leq \min\{i-1, k\}  \\[0.4em]
(iii) \hspace{3mm} \dfrac{-2 m_{1,j}}{m_{1,i-1}}m_{i-1,i} & & \text{if} \ 3 \leq i \leq k+1 \ \land \ j \geq i \\[1.2em]
(iv) \hspace{3mm} m_{i,j}^{2} - \sum\limits_{t=3}^{k}\frac{m_{1,i}}{m_{1,t}}m_{t,j}^{t-1}& & \text{if} \  i \geq k+2 \ \land \ j \geq k+1
\end{array}\right.
\] \vspace{3mm}
\end{lemma}	

\begin{proof}	
We prove this by induction on $k$: \vspace{5mm}
		
Fix $k=3$ and lets $m_{i,j}^{3}$ denote the $ij$-entry in the matrix $M_{3}$.

\begin{enumerate}[(i)]
\itemsep3mm

\item If $i=1,2$ then $m_{i,j}^{3} = m_{i,j}^{2}$ by definition.

\item  Let $i \geq 3$ and $j \leq \min\{i-1, 3\}$.			
If $i=3$ then $j \leq 2$ and $m_{3,j}^{3} = m_{3,j}^{2} = 0$ by the form of $M_{2}$.			
If $i \geq 4$ then $j \leq 3$ and $m_{i,j}^{3} = m_{i,j}^{2} - \frac{m_{1,i}}{m_{1,3}}m_{3,j}^{2}$. If $j=1,2$ this last element is zero due to the form of $M_{2}$. If $j=3$ we have			
\begin{align}
m_{i,3}^{3} 
&= m_{i,3}^{2} - \frac{m_{1,i}}{m_{1,3}}m_{3,3}^{2} \notag \\
& =  m_{i,3}-\frac{m_{1,i}}{m_{1,2}}m_{2,3}-\frac{m_{2,i}}{m_{1,2}}m_{1,3} - \frac{m_{1,i}}{m_{1,3}}  \left( m_{3,3} - \frac{m_{1,3}}{m_{1,2}}m_{2,3} - \frac{m_{2,3}}{m_{1,2}}m_{1,3}\right)  \tag*{(\ref{sum})}\\
& =  m_{i,3} + \frac{m_{1,i}}{m_{1,2}}m_{2,3} - \frac{m_{2,i}}{m_{1,2}}m_{1,3} \notag \\ 
& =  \frac{(m_{i,3}m_{1,2} + m_{1,i}m_{2,3}) - m_{2,i}m_{1,3} }{m_{1,2}} \notag\\ 
&= \frac{m_{2,i}m_{1,3} - m_{2,i}m_{1,3} }{m_{1,2}}=0 \tag*{$[E_{1,2,3,i}]$}
\end{align}	

\item Let $3 \leq i \leq 4$ and $j \geq i$. If $i=3$ then $j \geq 3$ and

\begin{align}
m_{3,j}^{3}&= m_{3,j}-\frac{m_{1,3}}{m_{1,2}}m_{2,j}-\frac{m_{2,3}}{m_{1,2}}m_{1,j} =\frac{m_{3,j}m_{1,2} -(m_{1,3}m_{2,j} +m_{2,3}m_{1,j})}{m_{1,2}} \tag{\ref{sum}} \\ 
& = \frac{-2m_{1,j}}{m_{1,2}}m_{2,3} \tag*{$[E_{1,2,3,j}]$}
\end{align}
	
If $i=4$ then $j\geq 4$ and 

\begin{align}
& m_{4,j}^{3} = m_{4,j}^{2} - \frac{m_{1,4}}{m_{1,3}}m_{3,j}^{2} \notag \\
& =  m_{4,j} - \frac{m_{1,4}}{m_{1,2}}m_{2,j} - \frac{m_{2,4}}{m_{1,2}}m_{1,j} - \frac{m_{1,4}}{m_{1,3}}\left( m_{3,j} - \frac{m_{1,3}}{m_{1,2}}m_{2,j} - \frac{m_{2,3}}{m_{1,2}}m_{1,j}\right) \tag{\ref{sum}} \\
&=  m_{4,j} - \frac{m_{2,4}}{m_{1,2}}m_{1,j} - \frac{m_{1,4}m_{3,j}}{m_{1,3}} + \frac{m_{1,4}m_{2,3}}{m_{1,3}m_{1,2}}m_{1,j} \notag \\
&= \frac{m_{4,j}m_{1,3} - m_{1,4}m_{3,j}}{m_{1,3}} + \left( \frac{m_{1,4}m_{2,3} - m_{2,4}m_{1,3}}{m_{1,3}m_{1,2}}\right)m_{1,j} \tag*{$[E_{1,2,3,j} ,E_{1,2,3,4}]$}\\ 
&= -\frac{m_{1,j}m_{3,4}}{m_{1,3}} -\frac{m_{1,2}m_{3,4}}{m_{1,3}m_{1,2}}m_{1,j} = \frac{-2m_{1,j}}{m_{1,3}}m_{3,4} \notag
\end{align}
 \vspace{0mm}		
\item If $i \geq 5$ and $j \geq 4$ we have that $m_{i,j}^{3} = m_{i,j}^{2} - \frac{m_{1,i}}{m_{1,3}}m_{3,j}^{2}$. And this completes the proof for case $k=3$. \vspace{3mm}
\end{enumerate}

Suppose now that the claim is true up to $k \geq 3$. \vspace{2mm}
	
\begin{enumerate}[(i)]
\itemsep 0.5cm
\item If $i=1,2$ then $m_{i,j}^{k+1} = m_{i,j}^{k} = m_{i,j}^{2}$ by the induction hypothesis. 
\item Let $j \leq \min\{i-1, k+1\}$ and $i \geq 3$. 
If $j \leq \min\{i-1, k\}$ then 

$m_{i,j}^{k+1} = \left\{\begin{array}{l c l}
m_{i,j}^{k} & & \text{if} \ 3\leq i \leq k+1\\
m_{i,j}^{k} - \frac{m_{1,i}}{m_{1,k+1}}m_{k+1,j}^{k} & & \text{if} \ i \geq k+2
\end{array}\right. = 0$ \vspace{3mm} 

\hspace{-7mm} by the induction hypothesis. Let $j=k+1$ and $i \geq k+2$. Then 

\begin{align}
m_{i,k+1}^{k+1} &= 
m_{i,k+1}^{2} - \sum\limits_{t=3}^{k}\frac{m_{1,i}}{m_{1,t}}m_{t,k+1}^{t-1} -\frac{m_{1,i}}{m_{1,k+1}}\left( m_{k+1,k+1}^{2} - \sum\limits_{t=3}^{k}\frac{m_{1,k+1}}{m_{1,t}}m{t,k+1}^{t-1}\right) \tag{\ref{sum}} \\[0.5em]
& = m_{i,k+1}^{2} - \frac{m_{1,i}}{m_{1,k+1}}m_{k+1,k+1}^{2} \notag \\[0.5em]
& =  m_{i,k+1} - \frac{m_{1,i}}{m_{1,2}}m_{2,k+1} - \frac{m_{2,i}}{m_{1,2}}m_{1,k+1} - \frac{m_{1,i}}{m_{1,k+1}}\left(m_{k+1,k+1} - \frac{2m_{1,k+1}}{m_{1,2}}m_{2,k+1}\right) \tag{\ref{sum}}\\[0.5em]
& = m_{i,k+1} + \frac{m_{1,i}}{m_{1,2}}m_{2,k+1}- \frac{m_{2,i}}{m_{1,2}}m_{1,k+1}= \frac{m_{i,k+1}m_{1,2} + m_{1,i}m_{2,k+1} - m_{2,i}m_{1,k+1}}{m_{1,2}}= 0 \notag
\end{align}

\item Let $3 \leq i \leq k+2$ and $j \geq i$. 
If $3 \leq i \leq k+1$ then $m_{i,j}^{k+1} = m_{i,j}^{k} = \dfrac{-2 m_{1,j}}{m_{1,i-1}}m_{i-1,i}$.
If $i=k+2$ then $j \geq k+2$ and \begin{align}
& m_{k+2,j}^{k+1} = m_{k+2,j}^{k} - \frac{m_{1,k+2}}{m_{1,k+1}}m_{k+1,j}^{k} \notag \\ 
&= m_{k+2,j}^{2} - \sum\limits_{t=3}^{k}\frac{m_{1,k+2}}{m_{1,t}}m_{t,j}^{t-1} -  \frac{m_{1,k+2}}{m_{1,k+1}}\left(m_{k+1,j}^{2} - \sum\limits_{t=3}^{k}\frac{m_{1,k+2}}{m_{1,t}}m_{t,j}^{t-1} \right) \tag{\ref{sum}} \\
&= m_{k+2,j}^{2} - \frac{m_{1,k+2}}{m_{1,k+1}}m_{k+1,j}^{2} \notag \\
& = m_{k+2,j} - \frac{m_{1,k+2}}{m_{1,2}}m_{2,j} -\frac{m_{2,k+2}}{m_{1,2}}m_{1,j} - \frac{m_{1,k+2}}{m_{1,k+1}}\left(m_{k+1,j} - \frac{m_{1,k+1}}{m_{1,2}}m_{2,j}-\frac{m_{2,k+1}}{m_{1,2}}m_{1,j}\right) \tag{\ref{sum}}\\
& = m_{k+2,j}-\frac{m_{2,k+2}}{m_{1,2}}m_{1,j} - \frac{m_{1,k+2}m_{k+1,j}}{m_{1,k+1}} + \frac{m_{1,k+2}m_{2,k+1}}{m_{1,k+1}m_{1,2}}m_{1,j} \notag \\[0.5em]
& = \frac{m_{k+2,j}m_{1,k+1} - m_{1,k+2}m_{k+1,j}}{m_{1,k+1}} + \frac{m_{1,k+2}m_{2,k+1}-m_{2,k+2}m_{1,k+1}}{m_{1,k+1}m_{1,2}}m_{1,j} \notag \\[0.5em]
& = \frac{-m_{1,j}m_{k+1,k+2}}{m_{1,k+1}} + \frac{-m_{1,2}m_{k+1,k+2}}{m_{1,k+1}m_{1,2}}m_{1,j} \hspace{3.5cm} [E_{1,k+1,k+2,j} \ \text{and} \ E_{1,2,k+1,k+2}] \notag \\[0.5em]
& = \frac{-2m_{1,j}}{m_{1,k+1}}m_{k+1,k+2} \notag
\end{align}

\item If $i\geq k+3$ and $j \geq k+2$ then $$m_{i,j}^{k+1} = m_{i,j}^{k} -  \frac{m_{1,i}}{m_{1,k+1}}m_{k+1,j}^{k} = m_{i,j}^{2} -  \sum\limits_{t=3}^{k}\frac{m_{1,i}}{m_{1,t}}m_{t,j}^{t-1} -  \frac{m_{1,i}}{m_{1,k+1}}m_{k+1,j}^{k} = m_{i,j}^{2} - \sum\limits_{t=3}^{k+1}\frac{m_{1,i}}{m_{1,t}}m_{t,j}^{t-1}$$ 
\end{enumerate}
and this completes the induction
\end{proof} \vspace{2mm}
	
\begin{proof}[Proof of Proposition~\ref{Prop:Triangulated form}]
The claim of the form of $T_M$ then follows from Lemma~\ref{Lemma:Mk}: observe that case $(iv)$ in Lemma~\ref{Lemma:Mk} disappear for $k=n-1$ because $i$ can not be greater than $n+1$, so we are left with the first three cases, which are those of Proposition~\ref{Prop:Triangulated form} when we replace $k$ by $n-1$. 

For the second assertion, we observe from its definition that $M_0$ is obtained from $M$ by swapping its first two rows. In the other hand, if $1 \leq k \leq n-1$, $M_k$ is obtained by a sequence of row operations that do not alter the determinant, so we have that $det(M) = -det(M_0) = -det(M_k)$ for all $k \in \{1, \ldots, n-1\}$. In particular $det(M) = -det(M_{n-1}) = -det(T_M)$
\end{proof}	\vspace{3mm}

\section*{Acknowledgement} The author acknowledges the University of Leeds for its warm welcome. He would like to thank Karin Baur and Ana García Elsener for their helpful suggestions, as well as the helpful comments of the referees.

\printbibliography

\end{document}